\pdfoutput=1
\RequirePackage{ifpdf}
\ifpdf 
\documentclass[pdftex]{sigma}
\else
\documentclass{sigma}
\fi

\usepackage[all,cmtip]{xy}
\usepackage{enumitem}
\usepackage{subcaption}
\usepackage[labelsep=period,labelfont=bf,font=small]{caption}
\usepackage{mathdots}
\usepackage{tikz}

\def\End{{\rm End}}

\def\oM{\overline{\mathcal{M}}}

\def\C{\mathbb{C}}

\def\b1{{\mathbf 1}}

\def\E{\mathrm{E}}
\def\n{\mathrm{n}}
\def\L{\mathrm{L}}
\def\V{\mathrm{V}}

\def\g{\mathrm{g}}
\def\rarr{\rightarrow}
\def\W{\mathsf{W}}
\def\D{\mathsf{D}}

\def\rbig{10mm}
\def\rsmall{6mm}

\numberwithin{equation}{section}

\newtheorem{Theorem}{Theorem}[section]
\newtheorem*{Theorem*}{Theorem}

\newtheorem{Proposition}[Theorem]{Proposition}
 { \theoremstyle{definition}
\newtheorem{Definition}[Theorem]{Definition}

\newtheorem{Example}[Theorem]{Example}
\newtheorem{Remark}[Theorem]{Remark} }

\begin{document}


\newcommand{\arXivNumber}{2112.10182}

\renewcommand{\PaperNumber}{088}

\FirstPageHeading

\ShortArticleName{On the Picard Group of the Moduli Space of Curves via $r$-Spin Structures}

\ArticleName{On the Picard Group of the Moduli Space of Curves\\ via $\boldsymbol{r}$-Spin Structures}

\Author{Danil GUBAREVICH~$^{\rm ab}$}

\AuthorNameForHeading{D.~Gubarevich}

\Address{$^{\rm a)}$~Laboratoire de Math\'ematiques de Versailles, UFR des Sciences, Universit\'e de Versailles\\
\hphantom{$^{\rm a)}$}~Saint-Quentin en Yvelines, 45 avenue des \'Etats-Unis, 78035 Versailles, France}
\EmailD{\href{mailto:danil.gubarevich@uvsq.fr}{danil.gubarevich@uvsq.fr}}

\Address{$^{\rm b)}$~Faculty of Mathematics, National Research University Higher School of Economics,\\
\hphantom{$^{\rm b)}$}~6 Usacheva Str., 119048 Moscow, Russia}

\ArticleDates{Received January 05, 2022, in final form August 27, 2024; Published online October 06, 2024}

\Abstract{In this paper, we obtain explicit expressions for Pandharipande--Pixton--Zvon\-kine relations in the second rational cohomology of $\oM_{g,n}$ and comparing the result with Arbarello--Cornalba's theorem we prove Pixton's conjecture in this case.}


\Classification{14H10; 14N35}

\section{Introduction}

The study of cohomology of the moduli space of curves was initiated by Mumford \cite{M} in the~1980s. The most progress has been made in understanding the subring $R^*(\mathcal{M}_{g})$ of tautological classes. A systematic study by Faber and Zagier of the algebra of $\kappa$ classes on the moduli space $\mathcal{M}_{g}$ of nonsingular genus g curves led to a conjecture in 2000 of a concise set FZ of relations among~$\kappa$ classes in
\[
R^*(\mathcal{M}_{g}) \subset A^*(\mathcal{M}_{g}).
\]
 There are several proofs of the Faber--Zagier conjecture, one of them is given by Pandharipande and Pixton in~\cite{PP}, exploiting moduli space of stable quotients on $\mathbb{P}^1$.

Recent progress in 2012 has appeared in Pixton's paper \cite{Pix} about hypothetically the full set of relations in the tautological ring of $\oM_{g,n}$. Pixton's relations recover FZ when restricted to~${\mathcal{M}_{g}\subset \overline{\mathcal{M}_{g}}}$.

Using homogeneity of Witten's 3-spin class,
authors of \cite{PPZ} proved Pixton's relations hold in~$H^*\big(\oM_{g,n},\mathbb{Q}\big)$, proving Faber--Zagier conjecture in cohomology.
The set of Pixton's relations (in Chow or cohomology) is conjectured to be full.

In this paper, we explicitly obtain the Pandharipande--Pixton--Zvonkine (PPZ or also called $r$-spin) relations \cite{PPZ,PPZ16} in the group $RH^2\big(\oM_{g,n},\mathbb{Q}\big)$
(although we omit the easiest case $g=0$ for brevity). We show that the PPZ relations we obtained coincide with relations from Arbarello--Cornalba's result \cite[Theorem 2.2]{Cor}. In \cite{Cor}, the authors proved that their set of relations in~$H^2\big(\oM_{g,n},\mathbb{Q}\big)$ is full. Also
using the fact that PPZ (for $r$ = 3) relations coincide with Pixton's relations and hold in $H^*\big(\oM_{g,n}, \mathbb{Q}\big)$, proved in \cite{PPZ}, we prove Pixton's conjecture in cohomology for this codimension.

\section{Preliminaries}
\subsection[Tautological classes on M\_\{g,n\}]{Tautological classes on $\boldsymbol{\oM_{g,n}}$}

Throughout the article, we are working with cohomology with coefficients in $\mathbb{Q}$.

Let $\oM_{g,n}$ be the Deligne--Mumford compactification by stable curves of the moduli space of nonsingular complex genus $g$ curves with $n$ markings, see \cite{Z} for the survey. Let
$
\pi\colon \overline{\mathcal{C}}_{g,n} \to \overline{\mathcal{M}}_{g,n}
$
be the universal curve. Considering $\overline{\mathcal{C}}_{g,n}$ and $\overline{\mathcal{M}}_{g,n}$ as orbifolds is sufficient for our purposes.

The fiber of a cotangent line bundles $\mathbb{L}_i$, $i=1,\dots,n$ over $\oM_{g,n}$ over a point $x \in \oM_{g,n}$ is a~cotangent line to a curve $(C_x,x)\in \overline{\mathcal{C}}_{g,n}$ at $i$-th marked points.
Denote the first Chern class of the line bundles $\mathbb{L}_i$ by $\psi_i$ {$($\em $\psi$-classes$)$}
\begin{align*}
\psi_i = c_1(\mathbb{L}_i)\in H^2\big(\oM_{g,n},\mathbb{Q}\big).
\end{align*}
{\em $\kappa$-classes} are defined via the universal map $\pi\colon \oM_{g,n+1} \to \oM_{g,n}$,
\begin{align*}
\kappa_i = \pi_{*}\big(\psi^{i+1}_{n+1}\big)\in H^{2i}\big(\oM_{g,n},\mathbb{Q}\big).
\end{align*}

Natural maps between moduli spaces of curves include:
\begin{enumerate}\itemsep=0pt
\item[(i)] Maps forgetting the last m points and further stabilizing the curve
$
p_m \colon \oM_{g, n+m} \to \oM_{g,n}$.
\item[(ii)] Gluing maps of separating kind
$
r\colon \oM_{g_1, n_1+1} \times \oM_{g_2, n_2+1} \to \oM_{g,n}$
which identify the point with label $n_1+1$ on the first curve with point label $n_2+1$ on the second curve.

\item[(iii)] Gluing maps of nonseparating kind
$
q \colon \oM_{g-1, n+2} \to \oM_{g,n}$
which identifies points with labels $n+1$ and $n+2$.
\end{enumerate}

Let
$
RH^*\big(\oM_{g,n}\big) \subset H^*\big(\oM_{g,n}\big)
$
be the subring of tautological
classes in the rational cohomology ring. It is defined as the smallest $\mathbb{Q}$-subalgebra closed under push-forwards by forgetting and gluing maps.

First examples of elements in $RH^*\big(\oM_{g,n}\big)$ include the following classes.
$1\in H^{0}\big(\oM_{g,n}\big)$ is tautological since $RH^*\big(\oM_{g,n}\big)$ is a ring with unit by definition. Classes represented by boundary strata and their intersections also lie in tautological ring since they are images under gluing maps.
It is clear that $\kappa$- and $\psi$-classes lie in $RH^*\big(\oM_{g,n}\big)$ from their definition and from relation $\psi_i = - \pi_{*}\big(D^{2}_{i}\big)$, where $D_i \in H^{2}\big(\oM_{g,n+1},\mathbb{Q}\big)$ is a class of a divisor whose generic point is represented by a nodal curve with two irreducible components of genera $g$ and $0$ and
$i$-th and~$(n+1)$-th points lie on a rational component.

The natural question are the rings $RH^*\big(\oM_{g,n}\big) \subset H^*\big(\oM_{g,n}\big)$ isomorphic is open but there was a significant progress from the pioneering studies of moduli of curves initiated by Mumford.

\subsection{Cohomological field theories} \label{cft}
We recall here the basic definitions of a cohomological field theory (CohFT) by Kontsevich and Manin, see \cite{PPP} for a survey.

It consists of the following data ($V$,$\eta$,$\b1$,$(c_{g,n})_{2g-2+n>0}$), where
\begin{enumerate}\itemsep=0pt
\item[(i)] $V$ is a finite-dimensional $\mathbb{Q}$-vector space,
\item[(ii)]$\eta$ is a non-degenerate symmetric bilinear 2-form (metric),
\item[(iii)]$\b1\in V$ is a distinguished element(unit vector).
\end{enumerate}

Choosing a basis $\{e_i\}$ of $V$, denote by $\eta_{jk}=\eta(e_j,e_k)$ the value of the metric on the basis vectors and by $\eta^{jk}$ the inverse matrix.
Given these data $(V,\eta,\b1)$ the $(c_{g,n})_{2g-2+n > 0}$ is a~collection of linear operators
$c_{g,n} \in H^*\big(\oM_{g,n}\big) \otimes (V^*)^{\otimes n}$.
It means that $c_{g,n}$ associate a (non-homogeneous)cohomology class $c_{g,n}(x_1\otimes \dots \otimes x_n)\! \in H^*\big(\oM_{g,n}\big)$ to a vector
$x_1\otimes \dots \otimes x_n \in V^{\otimes n}$.

The CohFT axioms imposed on $c_{g,n}$ are the following.

\begin{Definition}\label{defcohft}
A tuple ($V$,$\eta$,$\b1$,$(c_{g,n})_{2g-2+n>0}$) is called a {\em cohomological field theory $($with unit$)$} if it satisfies the following axioms:
\begin{enumerate}\itemsep=0pt
\item[(i)] For all $x_i \in V$,
\begin{align*}
c_{g,n}(x_1\otimes \dots \otimes x_n) = c_{g,n}(x_{\sigma(1)}\otimes \dots \otimes x_{\sigma(n)})
\end{align*}
for each $\sigma$ in the symmetric group $\Sigma_n$ acting on the set of marked points and simultaneously on
copies of $V$ in the tensor product $V^{\otimes n}$.
\item[(ii)] The pullback of $c_{g,n}$ under the gluing map $q$ of nonseparating kind equals the contraction of $c_{g-1,n+2}$ with
$\eta^{jk}\otimes e_j \otimes e_k$ inserted at the two identified points\footnote{Here and in the rest of the paper, we use the Einstein summation convention.}
\begin{align*}
q^*c_{g,n}(x_1\otimes \dots \otimes x_n) =
\sum_{j,k} \eta^{jk} c_{g-1,n+2}(x_1\otimes \dots \otimes x_n \otimes e_j \otimes e_k)
\end{align*}
in $H^*(\oM_{g-1,n+2},\mathbb{Q})$ for all $x_i \in V$.
The pullback of $c_{g,n}$ under a gluing map $r$ of separating kind equals the contraction $c_{g_1, n_1+1} \otimes c_{g_2, n_2+1}$ again with $\eta^{jk}\otimes e_j \otimes e_k$ inserted at the two identified points
\begin{gather*}
r^*(c_{g,n}(x_1\otimes \dots \otimes x_n)) \\
\qquad=\sum_{j,k} \eta^{jk} c_{g_1,n_1+1}(x_1\otimes \dots \otimes x_{n_1}\otimes e_j) \otimes
c_{g_2,n_2+1}(x_{n_1+1}\otimes \dots \otimes x_n \otimes e_k)
\end{gather*}
in $H^*\big(\oM_{g_1,{n_1+1}},\mathbb{Q}\big) \otimes
H^*\big(\oM_{g_2,{n_2+1}},\mathbb{Q}\big)$ for all $x_i \in V$.

\item[(iii)]The pullback of $c_{g,n}$ under the map $p$ forgetting the last point is required to satisfy
\begin{align*}
c_{g,n+1}(x_1\otimes \dots \otimes x_n \otimes \b1) = p^*c_{g,n} (x_1\otimes \dots \otimes x_n)\
\end{align*}
for all $x_i \in V$.
In addition, the equality
$
c_{0,3}(x_i\otimes x_j\otimes \b1) = \eta(x_i,x_j)
$
is required for all~${x_i \in V}$.
\end{enumerate}
\end{Definition}

\begin{Remark}\label{cohunit}
A tuple ($V$, $\eta$, $(c_{g,n})_{2g-2+n>0})$ is called a {\em cohomological field theory $($without unit$)$} if it satisfies properties (i) and (ii) above.
\end{Remark}

The degree zero part $\omega$ of a CohFT $c= (c_{g,n})_{2g-2+n>0}$ is called the \emph{topological part} of $c$
 \begin{align*}
 \omega_{g,n} = [c_{g,n}]^0 \in H^0\big(\oM_{g,n},\mathbb{Q}\big) \otimes
(V^*)^{\otimes n}.
 \end{align*}

Via property (ii) one sees that $\omega_{g,n}(x_1 \otimes \dots \otimes x_n)$
is determined by considering stable curves with a maximal number of nodes.
Such a curve is obtained by identifying several rational curves with three marked points. The value of $\omega_{g,n}(x_1 \otimes \dots \otimes x_n)$ is
thus uniquely specified by the values of $\omega_{0,3}$ and by the quadratic form~$\eta$. In other words, given $V$ and $\eta$, a topological part is uniquely determined by the associated quantum product, introduced in the next subsection.

Our main example of CohFT will be the CohFT associated with Witten's $r$-spin class and its shifted version.
\begin{Example}[Witten's $r$-spin class]
Let $r\geq2$ be an integer. For every $r$, there is a CohFT obtained from
Witten's $r$-spin class, introduced by Witten in \cite{Witten}. See \cite{P,PV} for an algebraic construction of Witten's top Chern class where
it was proved that the axioms of CohFT (i)--(iii) hold.

Witten's $r$-spin theory gives a family of classes
\begin{align*}
W^r_{g,n}(a_1, \dots, a_n) \in H^*\big(\oM_{g,n}\big) \end{align*}
for $a_1, \dots, a_n \in \{0, \dots, r-2 \}$.

Now given an $(r-1)$-dimensional $\mathbb{Q}$-vector space $V_r$ with basis
 $e_0, \dots, e_{r-2}$, metric
\begin{align*}
\eta_{ab} = \eta(e_a, e_b) =\delta_{a+b,r-2} ,
\end{align*}
and unit vector $\b1 = e_0$ one defines a CohFT $\W^r_{g,n}$ by
\begin{align*}
\W^r_{g,n}\colon\ V_r^{\otimes n} \rightarrow H^*\big(\oM_{g,n}\big),
\qquad
\W^r_{g,n}( e_{a_1} \otimes \dots \otimes e_{a_n}) =
W^r_{g,n}(a_1, \dots, a_n) .
\end{align*}
\emph{Witten's class} $W^r_{g,n}(a_1, \dots, a_n)$ has (complex) degree given
by the formula
\begin{align*}
\text{deg}_{\C} W^r_{g,n}(a_1, \dots, a_n) & =
\D^r_{g,n}(a_1, \dots, a_n) = \frac{(r-2)(g-1) + \sum_{i=1}^n a_i}{r} .
\end{align*}
If $\D^r_{g,n}(a_1, \dots, a_n)$ is not an integer, the corresponding
Witten's class vanishes.
\end{Example}

To formulate the Givental--Teleman classification result, we need the following notion of semisimplicity of a CohFT.

Let $(V,\eta,\b1)$ be a data associated with a CohFT $c_{g,n}$.
For $x_1, x_2\in V$, the {\em quantum product}~${x_1 \bullet x_2\in V}$ is uniquely determined by the condition:
for every $x_3\in V$
\begin{align*}
\eta(x_1 \bullet x_2, x_3) = c_{0,3}(x_1 \otimes x_2 \otimes x_3)\in \mathbb{Q} .\end{align*}
The quantum product $\bullet$ is commutative by CohFT axiom (i). Let us check that associativity of $\bullet$ follows from CohFT axiom (ii).
Further, we often write just $c_{0,3}(x_1, x_2, x_3)$ instead of~${c_{0,3}(x_1 \otimes x_2 \otimes x_3)}$ for short.
Fix a basis $\{e_i\}$ of $V$. Then the product in this algebra gives rise to structure constants $e_a \bullet e_b = c^{i}_{ab}e_i$, where the Einstein summation convention is assumed. Hence
\begin{align*}
\eta((e_a\bullet e_b)\bullet e_c,e_d) = c^{i}_{ab} \eta(e_i\bullet e_c, e_d)
=c^{i}_{ab} c^{j}_{ic} \eta_{jd}.
\end{align*}
Observe that $c^{i}_{ab} = \eta^{ic}c_{0,3}(e_a,e_b,e_c)$, hence we get
\begin{align*}
\eta^{ip}c_{0,3}(e_a,e_b,e_p)\eta^{jq}c_{0,3}(e_i,e_c,e_q)\eta_{jd},
\end{align*}
which is the same as $q^{*}c_{0,4}(e_a,e_b,e_c,e_d)$ by axiom (ii). On the other hand
\begin{align*}
\eta(e_a\bullet (e_b\bullet e_c),e_d)& = c^{i}_{bc} \eta(e_a\bullet e_i, e_d)
=c^{i}_{bc} c^{j}_{ai} \eta_{jd}
=\eta^{ip}c_{0,3}(e_b,e_c,e_p)\eta^{jq}c_{0,3}(e_a,e_i,e_q)\eta_{jd} \\
&= c_{0,3}(e_b,e_c,e_p)\eta^{ip}c_{0,3}(e_a,e_i,e_d),
\end{align*}
which is again $q^{*}c_{0,4}(e_a,e_b,e_c,e_d)$ by axioms (i) and (ii). Then from the nondegeneracy of $\eta$ the associativity follows. Moreover, a simple check shows that $(V,\bullet,\b1)$ is a Frobenius algebra, see~\cite{K03}.
A CohFT $c_{g,n}$ is called \emph{semisimple} if the algebra $(V,\bullet,\b1)$ is semisimple, i.e., the complexification~${V \otimes_{\mathbb{Q}} \mathbb{C}}$ has a basis $\{e_i\}$ of idempotents
$e_i\bullet e_j = \delta_{ij}e_i$.
Since $\W^r_{g,n}$ itself is not semisimple, we will need its shifted version.
\begin{Example}[shifted Witten's $r$-spin class]\label{swc}
Given a vector $\tau \in V_r$, the shifted Witten's class is defined by
\begin{align*}
\W_{g,n}^{r,\tau}(v_1 \otimes \cdots \otimes v_n) =
\sum_{m \geq 0}\frac{1}{m!}
p_{m*} \W^r_{g,n+m}\big(v_1 \otimes \cdots \otimes v_n \otimes \tau^{\otimes m}\big),
\end{align*}
where $p_m \colon \oM_{g,n+m} \to \oM_{g,n}$ is the forgetful map.
The \emph{shifted Witten's class} $\W^{r,\tau}$ determines a~CohFT with unit. Namely,
applying a map forgetting the last point, we get
\begin{align*}
p^{*}_1 \W_{g,n}^{r,\tau}(v_1 \otimes \cdots \otimes v_n) &
=\sum_{m \geq 0}\frac{1}{m!}
p_{m*}p^{*}_{1} \W^r_{g,n+m}\big(v_1 \otimes \cdots \otimes v_n \otimes \tau^{\otimes m}\big)\\
&=\sum_{m \geq 0}\frac{1}{m!} p_{m*}\W^r_{g,n+m+1}\big(v_1 \otimes \cdots \otimes v_n \otimes \tau^{\otimes m}\otimes \b1\big)\\
&=\W^r_{g,n+1}(v_1 \otimes \cdots \otimes v_n \otimes \b1).
\end{align*}
Moreover, let us choose a particular point $\widetilde{\tau}=(0,r\phi,0,\dots,0)=r\phi e_1\in V_r$ and check the semisimplicity of this CohFT.

Denote by $\langle a_1,\dots,a_n\rangle^{\tau}:= \int_{\oM_{0,n}}\W_{0,n}^{r,\tau}(e_{a_1} \otimes \cdots \otimes e_{a_n})$ the \emph{n-point function}.
Firstly, we use the known values for Witten's class
\begin{align*}
 \W_{0,3}(e_a,e_b,e_p)=\delta_{p,r-2-a-b}, \qquad \W_{0,4}(e_1,e_1,e_{r-2},e_{r-2})=\frac{1}{r}[pt]\in H^2\big(\oM_{0,4},\mathbb{Q}\big).
\end{align*}
Using \cite[Proposition 4.1]{PPZ16},
we compute the 3-point function
$\langle a,b,p\rangle^{\widetilde{\tau}}\! =\! \delta_{p,r-2-a-b}+\delta_{p,2r-3-a-b}\phi$.
Then we compute
\begin{align*}
e_{a} \bullet_{\widetilde{\tau}} e_{b} = c^{i}_{ab}e_i=\eta^{ip}\W_{0,3}^{r,\widetilde{\tau}}(e_a,e_b,e_p)e_i
=
\begin{cases}
\eta^{i, r-2-a-b} e_i = e_{a+b}& \text{if } a+b \leq r-2 ,\\
\phi \eta^{i, 2r-3-a-b}e_i= \phi e_{a+b-r+1} & \text{if } a+b \geq r-1.
\end{cases}
\end{align*}
Then changing a basis to $\widetilde{e}_a:=\phi ^{-a/(r-1)}e_a$ the new product is
\begin{align*}
\widetilde{e}_a \bullet_{\widetilde{\tau}} \widetilde{e}_b =
\phi^{-(a+b)/r-1} e_a \bullet e_b =
\begin{cases}
\widetilde{e}_{a+b} & \text{if } a+b\leq r-2,\\
\widetilde{e}_{a+b-r-1}& \text{if } a+b\geq r-1
\end{cases}
=
\widetilde{e}_{a+b \mod r-1}.
\end{align*}
Finally, after another change of basis $\bar{e}_i :=\sum^{r-2}_{a=0} \xi ^{a i} \widetilde{e}_a$, $\xi^{r-1} =1$, we derive a basis of idempo\-tents
\begin{align*}
\bar{e}_i \bullet_{\widetilde{\tau}} \bar{e}_j =
\sum_{a,b=0}^{r-2} \xi^{a i +b j}\widetilde{e}_{a+b \mod r-1}=
\sum_{a,c=0}^{r-2}\xi^{a(i-j)+c j}\widetilde{e}_j=
\sum^{r-2}_{a=0} \xi ^{a(i-j)}\bar{e}_{j} = (r-1)\delta_{ij}\bar{e}_{j},
\end{align*}
since if $i-j = x\neq 0$ then $1+\xi^{x}+\dots + \xi ^{(r-2)x} = 0$. So, $\{\bar{e}_{i}/(r-1)\}$ is a basis of idempotents.
\end{Example}

Now we recall how to assign to the given CohFT $(V,\eta,\mathbf{1},(c_{g,n})_{2g-2+n>0})$ a Frobenius manifold structure on $V$. Choose a basis $V=\mathbb{C}\langle e_1,\dots, e_N\rangle$ and a distinguished vector $\mathbf{1} = e_1$.
A~full CohFT potential $\mathcal{F}\in \mathbb{C}[[t^{*,*},\hbar]]$ is the following formal power series in $t^{d_k,i_k}$, $d_k \geq 0$, $i_k \in \{1,\dots,N\}$:
\begin{align*}
 \mathcal{F}\big(\big\{t^{d_k,i_k}\big\},\hbar\big) := \sum_{g\geq0}\hbar^{g-1}\mathcal{F}_g,
\end{align*}
where
\begin{align*}
 \mathcal{F}_g\big(\big\{t^{d_k,i_k}\big\}\big) := \sum_{n\geq x(g)}\frac{1}{(n-x(g))!}\sum_{\substack{d_1,\dots,d_n \\ i_1,\dots,i_n}} \int_{\oM_{g,n}}c_{g,n}(e_{i_1},\dots,e_{i_n})\psi_1^{d_1}\cdots \psi_n^{d_n}t^{d_1,i_1}\cdots t^{d_n,i_n}.
\end{align*}
The stability condition is captured by{\samepage
\begin{align*}
x(g):=
\begin{cases}
3 & \text{if} \ g=0 ,\\
1 & \text{if} \ g=1 ,\\
0 & \text{if} \ g\geq 2 .
\end{cases}
\end{align*}
Note that the intersection numbers above are nonzero only in degrees $3g-3+n$.}

A \emph{Frobenius potential} $F \in \mathbb{C}\big[\big[t^1,\dots,t^N\big]\big]$ is defined as
\begin{align*}
 F\big(t^1,\dots,t^N\big) :={}& \mathcal{F}_0\big(\big\{t^{d_k,i_k}\big\}\big)|_{\{t^{d_k,*}=0, d_k>0\}}\\
={}&
 \sum_{n\geq 3}\frac{1}{(n-3)!}\sum^N_{i_1,\dots,i_n=1} \int_{\oM_{0,n}}c_{0,n}(e_{i_1},\dots,e_{i_n}) t^{i_1}\cdots t^{i_n},
\end{align*}
where we put $t^{\mu}:=t^{0,\mu}$.

It gives the structure of \emph{formal Frobenius manifold} with flat metric
$
 \eta_{\mu\nu} = \frac{\partial^3 F}{\partial t^1\partial t^{\mu}\partial t^{\nu}}
$
and a~Frobenius algebra structure with multiplication $e_{\alpha} \bullet e_{\beta}:=c_{\alpha \beta \gamma}\eta^{\gamma \theta} e_{\theta}$, where
\begin{align*}
 c_{\alpha \beta \gamma}(t^*)= \frac{\partial^3 F}{\partial t^{\alpha}\partial t^{\beta}\partial t^{\gamma}}, \qquad \big(\eta^{\alpha\beta}\big):= (\eta_{\alpha\beta})^{-1}.
\end{align*}
CohFT axioms force $F$ to satisfy the WDVV equation
$
 c_{\alpha \beta \theta} \eta^{\theta\delta}
 c_{\delta \gamma \rho}=
 c_{\alpha \gamma \theta} \eta^{\theta\delta}
 c_{\delta \beta \rho}$.

\subsection{Reconstruction}\label{GTC}
Let $c=(c_{g, n})_{2g-2+n>0}$ be a CohFT, not necessarily with unit,
associated with data $(V, \eta)$.
Let~$R$ be a matrix series
\begin{align*}
R(z)= \sum_{k=0}^\infty R_k z^k
\in \mathrm{Id} +z\cdot \operatorname{End}(V)[[z]],
\end{align*}
which satisfies the symplectic condition
$\eta(R(z)x,R(-z)y)=\eta(x,y)$
for all $x$, $y$ in $V$. The group of such series acts on a space of CohFTs and we now recall the action, following closely \cite[Section~1.2.1]{PPP}.

One can define a
 new CohFT $Rc$
on the vector space $(V, \eta)$ is expressed as a sum over stable graphs $\Gamma$, see \cite[Section 0.2]{PPZ} for an exposition,
with summands given by
 products of vertex, edge, and leg contributions. Let $\mathsf{G}_{g,n}$ denote the set of all stable graphs
(up to isomorphism) of genus $g$ with $n$ legs. To each stable graph $\Gamma$, one can associate the moduli space
\smash{$
\oM_\Gamma =\prod_{v\in \V} \oM_{\g(v),\n(v)}
$}
with $2g(v)-2+n(v)>0$.
 There is a
canonical morphism
$\iota_{\Gamma}\colon \oM_{\Gamma} \rarr \oM_{g,n}$
 with image equal to the boundary stratum
associated to the graph $\Gamma$.

Then new CohFT $Rc$ is expressed as
\[
(Rc)_{g, n}=
\sum_{\Gamma\in\mathsf{G}_{g,n}} \frac{1}{|\mathrm{Aut}(\Gamma)|} \iota_{\Gamma \star} \biggl(
\prod_{v\in \V}\mathrm{Cont}(v)\prod_{e\in \E}\mathrm{Cont}(e)
\prod_{l\in \L}\mathrm{Cont}(l)
 \biggr) ,
 \]
 where we place
\begin{itemize}\itemsep=0pt
\item [(i)]
$\mathrm{Cont}(v)=c_{g(v), n(v)}$
at each vertex, where $g(v)$ and $n(v)$ denote the genus and number of half-edges and legs of the vertex,
\item [(ii)] the $\operatorname{End}(V)$-valued cohomology class
$\mathrm{Cont}(l)=R^{-1}(\psi_l)$
at each leg,
$\psi_l\!\in\! H^2\big(\oM_{g(v),n(v)},\mathbb{Q}\big)$
is the cotangent class at the marking corresponding to the leg,
\item [(iii)]
$\mathrm{Cont}(e)=\bigl(\eta^{-1}-R^{-1}(\psi'_e)\eta^{-1} R^{-1}(\psi''_e)^{\top}\bigr)/(\psi'_e+\psi''_e)$
at each edge,
where $\psi'_e$ and $\psi''_e$ are the cotangent classes\footnote{Note that if $\psi_e'$ or $\psi_e''$ corresponds to a marking on genus 0 component with less then four markings, then it is zero.} at the node which represents the edge $e$.
\end{itemize}
Then we can state Givental--Teleman classification result for CohFTs, proved initially in~\cite{T} for underlying Frobenius potentials.
Let $c$ be a semisimple CohFT with unit on $(V, \eta,\b1)$, and
let~$\omega$ be the topological part of $c$.
For a matrix $R$ satisfying the symplectic condition, define~${R.\omega = R(T(\omega))}$,
where $T\in V[[z]]$
is a series with no terms of degree $0$ or $1$
\begin{align*}
T(z)=z\big(\mathrm{Id}-R^{-1}(z)\big)\cdot \b1\in z^2V[[z]],
\end{align*}
acting on a CohFT $c_{g,n}$ by the formula
\begin{equation*}
(Tc)_{g, n} (x_1\otimes \dots \otimes x_n)=\sum_{m=0}^{\infty} \frac{1}{m!} p_{m*}
(c_{g, n+m}(x_1\otimes \dots \otimes x_n \otimes T(\psi_{n+1})\otimes \dots \otimes T(\psi_{n+m})))
 ,
\end{equation*}
where $p_{m}\colon \oM_{g, n+m}\to \oM_{g, n}$ is the morphism forgetting the last $m$ markings.

By \cite[Proposition 2.12]{PPZ}, $R.\omega$ is a CohFT with unit on $(V,\eta, \b1)$.
The Givental--Teleman classification asserts
the {\em existence} of a unique $R$-matrix which exactly
recovers $c$.

\begin{theorem*}[Teleman]
There exists a unique symplectic matrix
$R\in \mathrm{Id}+z\cdot \End (V)[[z]]$
which reconstructs $c$ from $\omega$,
$c=R.\omega $, as a CohFT with unit.
\end{theorem*}

\section[Tautological relations at semisimple point]{Tautological relations at semisimple point:\\ $\boldsymbol{\widetilde{\tau} = (0, r \phi, 0,\ldots, 0)}$}\label{tautrel}

We chose this point since at this point
one can derive the relatively simple expression for the topological field theory of the corresponding shifted Witten's CohFT, although the closed formulas for $R$-matrix are not known. Acting on the topological field theory by $R$-matrix, we successfully deduce the PPZ relations (also called $r$-spin relations)
 for $r\geq3$ in tautological ring, using degree vanishing argument. We will follow Pandharipande--Pixton--Zvonkine \cite{PPZ16}.

The shifted along the second basis vector $e_{1}\in V_r$ Witten's CohFT is
semisimple and has a~relatively simple expression for its topological part.
As we saw in Example~\ref{swc},
the quantum product at $\widetilde{\tau}$ is given by
\begin{align*}
\widetilde{e}_a \bullet_{\widetilde{\tau}}\widetilde{e}_b =
\begin{cases}
\widetilde{e}_{a+b} & \text{if} \ a+b \leq r-2 ,\\
\widetilde{e}_{a+b-r+1} & \text{if} \ a+b \geq r-1 .\\
\end{cases}
\end{align*}

By \cite[Proposition 4.2]{PPZ16}, the associated topological field theory takes the form
 \begin{align}
 \label{top}
\omega^{r,\widetilde{\tau}}_{g,n} (\widetilde{e}_{a_1}\otimes \cdots \otimes \widetilde{e}_{a_n})
= \phi^{(g-1)\frac{r-2}{r-1}}(r-1)^g \cdot \delta ,
\end{align}
where $\delta$ equals~$1$ if $g-1-\sum_{i=1}^n a_i$ is divisible by $r-1$ and~$0$ otherwise.
The idea is the following. Firstly, from definition of $\widetilde{e}_a$ observe that
\begin{align*}
\omega^{r,\widetilde{\tau}}_{0,3}(\widetilde{e}_{a},\widetilde{e}_{b},\widetilde{e}_{c}) =
\begin{cases}
\phi^{-\frac{r-2}{r-1}} & \text{if} \ a+b+c =-1 \mod r-1 ,\\
0 & \text{else},
\end{cases}
\end{align*}
and hence \smash{$\eta_{ab}=\phi^{-\frac{r-2}{r-1}}\delta_{a+b,r-2}$} in a basis $\{\tilde{e}_i\}$.
Then \smash{$\omega^{r,\widetilde{\tau}}_{g,n}$} for general $g$ and $n$ is the same as the restriction of \smash{$\Omega^{r,\widetilde{\tau}}_{g,n}$} to a maximally degenerate curve $[C]\in \oM_{g,n}$ with $2g-2+n$ rational components and $3g-3+n$ nodes.
From this, 3-valent dual graph we see that the sum of the insertions to $\omega^{r,\widetilde{\tau}}_{0,3}$ from $\left\{0,\dots,r-2\right\}$ plus one is divisible by $r-1$ and the sum of the insertions on each edge is $r-2$. From this, it is clear that $g-1+\sum a_i$ must be divisible by
$r-1$. We~place an arbitrary insertion on a single branch of a node at every independent cycle of the dual graph, then the other insertions are uniquely determined. This gives us a factor $(r-1)^g$. The factor \smash{$\phi^{(g-1)\frac{r-2}{r-1}}$} comes from matching nodes and
rational component powers of $\phi$.

There is an approach for computing $R$-matrix recursively explained in \cite[Proposition~4.4]{PPZ16}. The upshot is the explicit formula for the coefficients.

In a basis $\widetilde{e}_0,\dots,\widetilde{e}_{r-2}$ the $R$-matrix $(R_m)^b_a$, $m\geq 0$, $a,b \in\{0,\dots,r-2\}$ has coefficients
\begin{align*}
(R_m)^b_a = \bigl[-r(r-1)\phi^{\frac{r}{r-1}}\bigr]^{-m} P_m(r,r-2-b) \qquad
 \mbox{if} \quad b+m = a \mod{r-1},
\end{align*}
and 0 otherwise.
The inverse matrix $R^{-1}(z)$ has coefficients
\begin{align*}
\big(R_m^{-1}\big)^b_a = \big[r(r-1)\phi^{\frac{r}{r-1}}\big]^{-m} P_m(r,a) \qquad
 \mbox{if} \quad b+m = a \mod{r-1},
 \end{align*}
 and 0 otherwise.
Here the polynomials $P_m(r,a)$ with initial condition $P_0(r,a) = 1$ are defined by the following recursive procedure. For $m \geq 1$,
\begin{align*}
P_m(r,a)={}&
\frac12\sum_{b=1}^a (2mr-r-2b) P_{m-1}(r,b-1)\nonumber\\
& - \frac1{4mr(r-1)}
\sum_{b=1}^{r-2} (r-1-b)(2mr-b)(2mr-r-2b) P_{m-1}(r,b-1) .\label{Eq:DefP}
\end{align*}
In fact, we will need only first two values
\begin{align*}
P_0 = 1, \qquad
P_1 = \frac12 a(r-1-a) - \frac1{24}(2r-1)(r-2).
\end{align*}

Now let $\mathsf{W}^{r,\widetilde{\tau}}$ be the cohomological field theory given by the shift of Witten's
$r$-spin class by the vector $\widetilde{\tau} = (0, r \phi, 0,\ldots, 0)$. Using the Givental--Teleman classification theorem, $\mathsf{W}^{r,\widetilde{\tau}}$ can be recovered as an $R$-matrix action on the topological field theory $\omega ^{r,\widetilde{\tau}}$ and by \cite[Theorem~9]{PPZ16} Witten's $r$-spin class $W^r_{g,n}(a_1,\dots,a_n)$ equals the part of \smash{$\mathsf{W}^{r, \widetilde{\tau}}_{g,n}$} of degree
\begin{align*}
\D^r_{g,n}(a_1, \dots, a_n) = \frac{(r-2)(g-1) + \sum_{i=1}^n a_i}r
\end{align*}
in $H^*\big(\overline{\mathcal{M}}_{g,n}\big)$.
The parts of \smash{$\mathsf{W}^{r, \widetilde{\tau}}_{g,n}$} of degree
higher than $\D^r_{g,n}$ vanish. This theorem is the key point to find relations.

Now we describe the procedure of finding $r$-spin relations at the point $\widetilde{\tau}$ from the shifted Witten's class. It will be convenient to specialize the decoration procedure introduced in Section~\ref{GTC} to the case of the reconstruction of the shifted Witten's CohFT $\mathsf{W}^{r, \widetilde{\tau}}_{g,n}$ from its topological part.

Fix $r\geq 3$. Fix $n$ numbers $0\leq a_1,\dots,a_n\leq r-2$. All constructions below depend on an auxiliary variable $\phi$ which keeps track
the codimension and we fix its exponent $d$. A {\em tautological $r$-spin relation} $T(g,n,r,a_1,\dots,a_n,d)=0$ depends on these choices, and it is obtained as~\smash{$T = r^{g-1}\sum_{k=0}^\infty \pi^{(k)}_* T_k/k!$}, where $T_k$ is the coefficient of $\phi^d$ in the expression in the decorated stable graphs of $\oM_{g,n+k}$ described below, and $\pi^{(k)}\colon \oM_{g,n+k}\to \oM_{g,n}$ forgets the last $k$ points.
Fix a stable graph $\Gamma \in \mathsf{G}_{g,n+k}$, and then equip
\begin{enumerate}\itemsep=0pt
\item[(i)] each vertex $v$ of $\Gamma$ with \smash{$\omega ^{r,\widetilde{\tau}}_{g(v),n(v)}$},

\item[(ii)]the first n legs by \smash{$\sum_{m=0}^\infty \big(R^{-1}_m\big)_{a_i}^b \psi_i^m \widetilde{e}_b$}, $i=1,\dots,n$,

\item[(iii)]the $k$ extra legs (the dilaton legs) with \smash{$-\sum_{m=1}^\infty \big(R^{-1}_m\big)_{0}^b \psi^{m+1}_{n+i} \widetilde{e}_b$}, $i=1,\dots,k$,

\item[(iv)]each edge $e$, where we denote by $\psi'_e$ and $\psi''_e$ the $\psi$-classes on the two branches of the corresponding node, with
\[
\mathrm{Cont}(e)= \frac{\eta^{i'i''}-\sum_{m',m''=0}^\infty \big(R^{-1}_{m'}\big)_{j'}^{i'} \eta^{j'j''} \big(R^{-1}_{m''}\big)_{j''}^{i''} (\psi'_e)^{m'}(\psi''_e)^{m''} }{\psi'_e+\psi''_e} \widetilde{e}_{i'}\otimes \widetilde{e}_{i''}.
\]
\end{enumerate}
The summations above by repeating indices is assumed.
Then $T_k$ is defined as the sum over all decorated stable graphs obtained by the contraction of all tensors assigned to their vertices, legs, and edges, further divided by the order of the automorphism group of the graph.

Observe that the exponent $d$ of $\phi$ comes from many places, so let us collect the total degree.%
\begin{itemize}\itemsep=0pt
 \item Every time we fed a leg we get a factor
$\phi^{-m r/(r-1)}$ from $R_m^{-1}$ and if we are
interested in relations in the group $RH^{2D}\big(\oM_{g,n},\mathbb{Q}\big)$, then the sum of indices $\sum m_i$ of matrices $R_m^{-1}$ should be equal to $D$. Hence, we get a factor \smash{$\phi^{-\frac{rD}{r-1}}$}.

\item At every vertex $v$ equipped with \smash{$\omega ^{r,\widetilde{\tau}}_{g(v),n(v)}$}, we get a contribution $\phi ^{(g(v)-1)(r-2)/(r-1)}$ and the total contribution from all the vertices will be \smash{$\phi^{(\sum^{\nu}_v g(v)-\nu)\frac{r-2}{r-1}}=\phi^{g-1-\delta}$}, where $\nu$ is the number of irreducible components and $\delta$ is the number of nodes the curve $\Gamma$ has.

\item Since we wrote each $R_m$ in a basis $\{\widetilde{e}_i\}$ and fed the first $n$ legs by the primary fields $e_1,\dots,e_n$, we should change a basis to $\{\widetilde{e}_i\}$. Hence, we earn an additional factor \smash{$\phi^{\frac{1}{r-1}\sum_i a_i}$} after the change of basis \smash{$e_i = \phi^{\frac{i}{r-1}}\widetilde{e}_i$}.

 \item Dilaton legs will not contribute in the exponent of $\phi$ since they are fed with $e_0=\widetilde{e}_0$.

 \item Since each edge of $\Gamma$ contains an inverse metric matrix $\eta^{ij}$, it gives a factor \smash{$\phi^{\frac{r-2}{r-1}}$} and the total factor will be \smash{$\phi^{\frac{r-2}{r-1}\delta}$}. Recall that the number of edges of $\Gamma$ is the number of nodes on the corresponding curve.
 \end{itemize}

 Summing up, the total degree $d$ of variable $\phi$ for the graph $\Gamma$ in codimension $D$ satisfies
 \begin{align*}
 d(r-1) = \sum a_i + (g-1)(r-2) - r D.
 \end{align*}
 From the topological part expression (\ref{top}), we can say that
 \begin{align*}
 \sum a_i = g-1+D+x(r-1)
 \end{align*}
 for $x\geq 0$.
 Now if Witten's $r$-spin class vanishes when $D >\D^r_{g,n} $ it is easy to see that $d$ must be negative.

 Finally, observe that
\begin{align*}
d(r-1) = g-1+D+x(r-1) + (g-1)(r-2) - r D = (g - 1- D + x)(r - 1),
\end{align*}
and hence
$ d<0$ if and only if $D\geq g+x$.

This shows that Pixtons's relations do not allow us to find relations in small degrees directly. We are interested in relations in complex degree 1 and to get them we need firstly to find relations in cohomology of higher degrees of the space with larger number of points and then push them forward by the forgetful morphism. For example, we could find the relations in the group $RH^4\big(\oM_{2,n+1},\mathbb{Q}\big)$ and then push them forward by forgetting the last point. Thus, we would get relations in the group $RH^2\big(\oM_{2,n},\mathbb{Q}\big)$, but this approach is technically difficult and we proceed in another way in the genus 2 case. But the genus 1 case can be worked out directly.

\section{The case of marked genus 1 curves}
In this section by applying the PPZ construction to the genus 1 case, we derive tautological relations and compare them with Arbarello--Cornalba's result. In this case, the PPZ approach works in each positive degree since $D\geq g+x =1+x$ is valid for each $D\geq1$ if we choose~${x=0}$.
Firstly, we apply described construction to the space $\oM_{1,2}$ and then calculate linear relations in the space $\oM_{1,n}$.
As a warm-up, we discuss linear relations in the space $\oM_{1,2}$. Further, we fix ${\phi=1}$ for simplicity.
The cohomology group $RH^2\big(\oM_{1,2},\mathbb{Q}\big)$ is generated by $\psi_1$, $\psi_2$, $\kappa_1$, $\delta_{\rm sep}$, $\delta_{\rm nonsep}$ with relations
\begin{align*}
\psi_1 = \psi_2,\qquad
2\psi_1 =\delta_{\rm sep}+\kappa_1,\qquad
 \psi_1 =\dfrac{1}{12}\delta_{\rm nonsep}+\delta_{\rm sep},
\end{align*}
see \cite[Theorem 2.2]{Cor}.

\begin{Example}
For $r\geq3$, the PPZ relations between the generators of the group $RH^2\big(\oM_{1,2},\mathbb{Q}\big)$ generate the set of relations from above.
\end{Example}

\begin{proof}
We have $D=1$, $x=0$, $g=1$, $a_1+a_2 = 1$; the degree 1 part of stable graph expressions gives us tautological relations since $1=D\geq g+x =1$.
Generators of this group are $\psi_1$, $\psi_2$, $\kappa_1$, $\delta_{\rm sep}$, $\delta_{\rm nonsep}$, where $\delta_{\rm sep}$ and $\delta_{\rm nonsep}$ are Poincar\'e dual classes to the generic point of divisor strata, see Figure~\ref{fig:chefigura}. The graphs with more then 1 edge will not contribute by dimensional reasons.

\begin{figure}[t]\centering
	\begin{subfigure}{0.25\textwidth}\centering
		\begin{tikzpicture}\label{I}
		
		\node[draw,circle,minimum size=\rbig] (g) at (0,0) {\( 1\)} ;
		
		\node[] (1) at (-0.6,-1.5) {\( a_1 \)};
		\node[] (n) at (0.6,-1.5) {\( a_2\)};
		
		\path (g) edge (1);
		\path (g) edge (n);
		\end{tikzpicture}
		\renewcommand{\thesubfigure}{i}
		\caption{smooth part}
	\end{subfigure}
\quad
	\begin{subfigure}{0.25\textwidth}\centering
		\begin{tikzpicture}\label{II}
		\node[draw,circle,minimum size=\rbig] (g) at (0,0) {\( 1 \)} ;
		\node[draw,circle,minimum size=\rsmall] (0) at (1.5,0) {\( 0\)};
	
		\node[] (1) at (3,-0.5) {\( a_1\)};
		\node[] (2) at (3,0.5) {\( a_2 \)};
		\path (g) edge (0);
		
		\path (0) edge (1);
		\path (0) edge (2);
		\end{tikzpicture}
	\renewcommand{\thesubfigure}{ii}
		\caption{\( \delta_{\rm sep} \)}
	\end{subfigure}
\quad
	\begin{subfigure}{0.25\textwidth}\centering
		\begin{tikzpicture}\label{III}
		\node[draw,circle,minimum size=\rbig] (0) at (3,0) {\( 0 \)} ;
	
		\node[] (1) at (2.4,-1.5) {\( a_1 \)};
		\node[] (2) at (3.6,-1.5) {\( a_2\)};
		\path (0) edge [loop above] node {} (0);
		
		\path (0) edge (1);
		\path (0) edge (2);
		\end{tikzpicture}
		\renewcommand{\thesubfigure}{iii}
		\caption{\( \delta_{\rm nonsep} \)}
	\end{subfigure}

\caption{The graphs of strata of \texorpdfstring{$\oM_{1,2}$}{M 1,2} of codimension at most 1.}	\label{fig:chefigura}
\end{figure}

Our stable graphs will contain 0, 1 or 2 dilaton leaves, so we will consider these cases separately. The smooth part gives

\begin{gather*}
 \left[\sum_{m=0}^{\infty}\sum^{r-2}_{b_1,b_2=0} \big(R_m^{-1}\big)^{b_1}_{a_1}\psi_1^m \sum_{m=0}^{\infty} \big(R_m^{-1}\big)^{b_2}_{a_2}\psi_2^m \omega ^{r,\widetilde{\tau}}_{1,2}(\widetilde{e}_{b_1},\widetilde{e}_{b_2})\right]_2\\
\qquad=\left[\sum_{b_1,b_2=0}^{r-2}\big(\big(R_0^{-1}\big)^{b_1}_{a_1}+\big(R_1^{-1}\big)^{b_1}_{a_1}\psi_1\big)
\big(\big(R_0^{-1}\big)^{b_2}_{a_2}+\big(R_1^{-1}\big)^{b_2}_{a_2}\psi_2\big) \omega ^{r,\widetilde{\tau}}_{1,2}(\widetilde{e}_{b_1},\widetilde{e}_{b_2})\right]_2\\
\qquad=(r-1)\big(\big(R_1^{-1}\big)^{a_1-1}_{a_1}\psi_1+\big(R_1^{-1}\big)^{a_2-1}_{a_2}\psi_2\big).
\end{gather*}
Since the factor $\bigl(r(r-1)\phi^{r/(r-1)}\bigr)^{-1}$ comes from each graph, we will omit it.

Let us develop the edge factor\footnote{See Section \ref{tautrel}.} in more detail up to linear terms in the $\psi$-classes.
\begin{align*}
\mathrm{Cont}(e)(\psi'_e+\psi''_e) &= \left[\eta^{i'i''}-\sum_{m',m''=1}^\infty \big(R^{-1}_{m'}\big)_{j'}^{i'} \eta^{j'j''} \big(R^{-1}_{m''}\big)_{j''}^{i''} (\psi'_e)^{m'}(\psi''_e)^{m''}\right]\widetilde{e}_{i^\prime}\otimes\widetilde{e}_{i^{\prime\prime}}\\
&=\bigl[-\big(R_1^{-1}\big)^{i''}_{r-2-i'}(\psi'+\psi'') + \text{higher degree terms}\bigr]\widetilde{e}_{i^\prime}\otimes\widetilde{e}_{i^{\prime\prime}} .
\end{align*}

The second and third graphs give
\begin{gather*}
 -\delta_{\rm sep}\sum_{b_1,b_2=0}^{r-2}\big(\big(R_0^{-1}\big)^{b_1}_{a_1}+\big(R_1^{-1}\big)^{b_1}_{a_1}\psi_1\big)
\big(\big(R_0^{-1}\big)^{b_2}_{a_2}+\big(R_1^{-1}\big)^{b_2}_{a_2}\psi_2\big)\big(R_1^{-1}\big)^{i''}_{r-2-i'},\\
\omega ^{r,\widetilde{\tau}}_{0,3}(\widetilde{e}_{b_1},\widetilde{e}_{b_2},\widetilde{e}_{i'})\omega ^{r,\widetilde{\tau}}_{1,1}
(\widetilde{e}_{i''})
=-(r-1)\delta_{\rm sep}\big(R_1^{-1}\big)^{0}_{1},\\
 -\delta_{\rm nonsep} \sum_{b_1,b_2=0}^{r-2} \big(R_0^{-1}\big)^{b_1}_{a_1}\big(R_0^{-1}\big)^{b_2}_{a_2}
 \sum_{i=0}^{r-2}\big(R_1^{-1}\big)^{i'}_{r-2-i''}
\omega ^{r,\widetilde{\tau}}_{0,4}(\widetilde{e}_{b_1},\widetilde{e}_{b_2},\widetilde{e}_{i'},\widetilde{e}_{i''}) \phi^{(r-2)/(r-1)}.
\end{gather*}

 Given a graph $\Gamma$ with a dilaton leg, the vertex carrying this leg will support a class of positive degree after push-forward. Thus a codimension 1 class can only be obtained if the graph $\Gamma$ was already trivial, see Figure~\ref{fig:SHefigura}.

\begin{figure}[t]\centering
	\begin{subfigure}{0.3\textwidth}\centering
		\begin{tikzpicture}\label{IV}
		
		\node[draw,circle,minimum size=\rbig] (g) at (0,0) {\( 1\)} ;
		
		\node[] (1) at (-0.6,-1.5) {\( a_1 \)};
		\node[] (n) at (0.6,-1.5) {\( a_2\)};
		\node[] (u) at (0.0,1.5) {};

		\path (g) edge (1);
		\path (g) edge (n);
		\draw (g) --(u) [dashed];
		\end{tikzpicture}
		\renewcommand{\thesubfigure}{iv}
		\caption{}
	\end{subfigure}

\caption{The remaining contributing graph with a dilaton leg.}
	\label{fig:SHefigura}
\end{figure}

The fourth graph gives
 \begin{gather*}
-\big(\big(R_0^{-1}\big)^{b_1}_{a_1}+\big(R_1^{-1}\big)^{b_1}_{a_1}\psi_1+(R_2^{-1})^{b_1}_{a_1}\psi_1^2\big)
\big(\big(R_0^{-1}\big)^{b_2}_{a_2}+\big(R_1^{-1}\big)^{b_2}_{a_2}\psi_1+(R_2^{-1})^{b_2}_{a_2}\psi_1^2\big),\\
\big(\big(R_1^{-1}\big)^b_0 \kappa_1+(R_2^{-1})^b_0\kappa_2\big)\omega ^{r,\widetilde{\tau}}_{1,3}(\widetilde{e}_{b_1},\widetilde{e}_{b_2},
\widetilde{e}_{b})=-(r-1)\big(R_1^{-1}\big)^{r-2}_0\kappa_1.
\end{gather*}

So, summing up we get a relation
\begin{gather*}
(r-1)[P_1(r,a_2)\psi_2+P_1(r,a_1)\psi_1]-\dfrac{1}{24}\big(2-3r+r^2\big)\delta_{\rm nonsep}, \\
- (r-1)P_1(r,1)\delta_{\rm sep}-(r-1)P_1(r,0)\kappa_1 = 0.
\end{gather*}

This equality must hold for every $r\geq3$, so collecting degrees of a variable $r$, we get following relations:
\begin{gather*}
\psi_1 = \psi_2,\qquad
2\psi_1 =\delta_{\rm sep}+\kappa_1,\qquad
 \psi_1 =\dfrac{1}{12}\delta_{\rm nonsep}+\delta_{\rm sep}.\tag*{\qed}
\end{gather*} \renewcommand{\qed}{}
\end{proof}

Similarly, we derived all PPZ relations in $RH^2\big(\oM_{1,n},\mathbb{Q}\big)$ for all $r$ simultaneously, and prove that they in fact are the full set of relations in $H^2\big(\oM_{1,n},\mathbb{Q}\big)$. In fact, we prove a stronger statement that the PPZ relations for every particular $r\geq3$ generate the full set of relations.

The group $H^2\big(\oM_{1,n},\mathbb{Q}\big)$ is generated by
$\psi_1$, $\dots$, $\psi_n$, $\kappa_1$, $\delta^{0}_{J}$, $\delta_{\rm nonsep}$ with relations
\begin{gather}
12\psi_i = \delta_{\rm nonsep} + 12\sum_{\substack{J\subset \{1,\dots,n\}|i\in J \\ |J|\geq 2}} \delta^{0}_{J}, \quad i=1,\dots,n,\qquad
 \kappa_1 = \sum^{n}_{i=1}\psi_i - \sum_{J\subset \{1,\dots,n\}}\delta^{0}_{J},\label{genus1}
\end{gather}
see \cite[Theorem 2.2]{Cor}.

\begin{Proposition}
For every particular $r\geq3$, the PPZ relations between the generators of the group $H^2\big(\oM_{1,n},\mathbb{Q}\big)$ generate the set of relations from above.
\end{Proposition}

\begin{proof}
 There are
 \begin{align*}
 n+1+\binom{n}{2}+\dots+\binom{n}{n}+1 = 2^n+1
 \end{align*}
 generators $\psi_1$, $\dots$, $\psi_n$, $\kappa_1$, $\delta^{0}_{J}$, $\delta_{\rm nonsep}$.
 Here $\delta^{0}_{J}$ is the class of the divisor representing a curve with markings from $J\subset \{1,\dots, n\}$, $|J|\geq 2$ on the genus~0 component.
Having $n$ markings, we have $n$ relations coming from $n$ different choices of a vector $(a_1,\dots,a_n)$ with $\sum^{n}_{i=0}a_i= 1$. Further, in the arguments of the topological part, we write just $b_i$ instead of $\widetilde{e}_{b_i}$ for short.
The contribution of the smooth part has the form
\begin{gather*}
\sum_{i=1}^{n}\big[\big(R^{-1}_{0}\big)^{a_i}_{a_i}+\big(R^{-1}_{1}\big)^{a_i-1}_{a_i}\psi_i\big]\omega_{1,n}(0,\dots,0)\\
\qquad=(r-1)\big[\psi_1 \big(R^{-1}_{1}\big)^{a_1-1}_{a_1}+\dots+\psi_n \big(R^{-1}_{1}\big)^{a_n-1}_{a_n}\big].
\end{gather*}
The contribution of a graph $\delta_J^0$ with separating node has the form
\begin{gather*}
-\delta^{0}_{J}\prod^{n}_{i=1}\big[\big(R^{-1}_{0}\big)^{a_i}_{a_i}+\big(R^{-1}_{1}\big)^{a_i-1}_{a_i}\psi_i\big]\sum^{r-2}_{p,q=0}\omega_{1, |I|+1}(a_{i_1},\dots,a_{i_{|I|}},p),\\
\omega_{0, |J|+1}(a_{j_1},\dots,a_{j_{|J|}},q)\big(R^{-1}_{1}\big)^{p}_{r-2-q},
\end{gather*}
where
\begin{gather*}
\omega_{1, |I|+1}(a_{i_1},\dots,a_{i_{|I|}},p) = (r-1)
\begin{cases}
1 & \text{if } a_I+p = 0 \mod r-1,\\
0 & \text{otherwise},
\end{cases}
\\
\omega_{0, |J|+1}(a_{j_1},\dots,a_{j_{|J|}},q) =
\begin{cases}
1 & \text{if } 1+a_J+q = 0 \mod r-1 ,\\
0 & \text{otherwise} .
\end{cases}
 \end{gather*}
From this, we immediately see that each
\smash{$
\big(R^{-1}_{1}\big)^{p}_{r-2-q} = \big(R^{-1}_{1}\big)^{-a_I}_{a_J}
$}
is nonzero since $r-2+a_I+1-1 = -a_J \mod r-1$ is true for every complementary subsets $I,J\subset \left\{1\dots n\right\}$. The total contribution over graphs with one separating node is
\begin{align*}
-\sum_{\substack{J\subset {1,\dots,n}\\|J|\geq 2}} \delta^{0}_{J}(r-1)\big(R^{-1}_{1}\big)^{-a_I}_{a_J},
\end{align*}
and specifying $a_i = 1$, $a_j = 0$, $i\neq j$, we would get
\begin{align*}
-(r-1)P_1(r,1)\sum_{J| i \in J} \delta^{0}_{J} -(r-1)P_1(r,0)\sum_{J| i \notin J} \delta^{0}_{J}.
\end{align*}

The contribution of a graph with a nonseparating node
 gives $ -\delta_{\rm nonsep}\frac{1}{24}\big(2-3r+r^2\big)$ as before.

The contribution of a graph with one dilaton leg has the form
\begin{align*}
-\prod^{n}_{i=1}\big[\big(R^{-1}_{0}\big)^{a_i}_{a_i}+\big(R^{-1}_{1}\big)^{a_i-1}_{a_i}\psi_i\big]\sum^{r-2}_{w=0}(R^{-1}_{0})^{w}_{0}\kappa_1 \omega_{1,n+1}(a_1,\dots,a_n,w),
\end{align*}
which is nonzero if and only if $w = r-2$ and equals $-\big(R^{-1}_{1}\big)^{r-2}_{0}(r-1)\kappa_1$.

So, for each fixed $i = 1,\dots,n$, we have the relation
\begin{gather*}
(r-1)\bigg[P_1(r,1)\psi_i +P_1(r,0) \sum_{j\neq i}\psi_j\bigg],\\
-(r-1)P_1(r,1)\sum_{J| i \in J} \delta^{0}_{J} -(r-1)P_1(r,0)\sum_{J| i \notin J} \delta^{0}_{J},\\
-P_1(r,0)(r-1)\kappa_1 -\delta_{\rm nonsep}\dfrac{1}{24}\big(2-3r+r^2\big) = 0,
\end{gather*}
which is a polynomial in $r$ of degree $3$. Taking coefficients and multiplying by 12 and 24, respectively, to clear denominators, we get
\begin{gather*}
\big[r^3\big]\colon\ \kappa_1 = \sum^{n}_{i=1}\psi_i - \sum_{J\subset \{1,\dots,n\}}\delta^{0}_{J},\\
\big[r^2\big]\colon\ 19\sum_{J| i \in J}\delta^{0}_{J} + 7\sum_{J| i \notin J}\delta^{0}_{J} + \delta_{\rm nonsep} + 7\kappa_1 = 19\psi_i + 7\sum_{j\neq i}\psi_j,
\end{gather*}
and eliminating kappa class, we find
\begin{gather*}
12\psi_i = \delta_{\rm nonsep} + 12\sum_{J|i\in J} \delta^{0}_{J}, \qquad
 i=1,\dots,n.
\end{gather*}
The relations coming from $[r]$, $[1]$ are consequences of relations coming from $\big[r^3\big]$, $\big[r^2\big]$.
Suppose now that we have another PPZ relation $Q=0$ in degree 1. Put $r=3$ for simplicity. As we noted at the end of Section~\ref{tautrel}, this relation came from a certain choice of a tuple $(a_1,\dots, a_n)$ such that~${\sum a_i = 1 \mod 2}$ and \smash{$1>\frac{1}{3}\sum a_i$}. There are only $n$ possibilities corresponding to a~choice of one $a_i=1$ and 0 the others. So, $Q$ is a linear combination of them.

Moreover, we claim that we get relations (\ref{genus1}) for every particular $r\geq3$, not just for simultaneously all $r$. For this, it is sufficient to express $\psi$-classes in terms of boundary divisor classes and use the linear independence of the latter, which follows from \cite[Theorem 2.2]{Cor}.

Namely, we rewrite the system of relations in the form
\begin{gather*}
(r-1)P_1(r,1)\psi_1+(r-1)P_1(r,0)\psi_2+\dots+(r-1)P_1(r,0)\psi_n -(r-1)P_1(r,0)\kappa_1 =B_1,
\\
\cdots\cdots\cdots\cdots\cdots\cdots\cdots\cdots\cdots\cdots
\\
(r-1)P_1(r,0)\psi_1 + \dots + (r-1)P_1(r,0)\psi_{n-1}\!+(r-1)P_1(r,1)\psi_n\!-(r-1)P_1(r,0)\kappa_1 =B_n,
\\
\psi_1 +\dots+ \psi_n - \kappa_1 = B_{n+1},
\end{gather*}
where the last equation is the coefficient of $r^3$ from above.
Here the terms $B_i$ are linear combinations of boundary divisors with coefficients depending on $r$.

The matrix $M$ on the left-hand side has determinant
\begin{align*}
\det M = (-1)^n(r-1)^{n}(P_1(r,0) - P_1(r,1))^{n-1}((n-1)P_1(r,0) + P_1(r,1)- nP_1(r,0)),
\end{align*}
and after further simplification it is seen that for $ r \geq 3$
\begin{align*}
\det M= -1\Bigg[\dfrac{(r-1)(r-2)}{2}\Bigg]^n \neq 0.
\end{align*}
This means that for every $r\geq3$ each $\psi_i$ can be expressed as a sum of divisor classes, without passing to taking coefficients $\big[r^i\big]$. Expressions of $\psi_i$ coming from $r\geq3$ and $r^{\prime}\geq3$ give a~linear relation on boundary divisor classes, which are linear independent. In other words,
we get the same relations for each particular $r\geq3$. Moreover, these relations are the same as in Arbarello--Cornalba's \cite[Theorem 2.2]{Cor}, where it is proved that it is the full set of relations in~$H^2\big(\oM_{1,n},\mathbb{Q}\big)$.
\end{proof}

\section{The case of marked genus 2 curves}
In this section, by applying the PPZ construction to the genus 2 case, we derive tautological relations and compare them with the Arbarello--Cornalba's result.

The group $RH^2\big(\oM_{2,n},\mathbb{Q}\big)$ is generated by $\kappa_1$, $\psi_1,\dots,\psi_n$ and divisor classes $\delta_0^I$, $\delta_1^I$, $I\subset\{1,\dots,n\}$, $\delta^{\prime}_{\rm irr}$ with one relation
\begin{align*}
	5\bigg(\kappa_1 - \sum \psi_i + \sum_{I\subset{1,\dots,n}}\delta_0^{I}\bigg) = \delta_{\rm irr}^{\prime} + 7\sum_{I\subset{1,\dots,n}}\delta_1^{I},
\end{align*}
see \cite[Theorem 2.2]{Cor}.

\begin{Proposition}
For every particular $r\geq3$, the PPZ relations between the generators of the group $RH^2\big(\oM_{2,n},\mathbb{Q}\big)$ generate the set of relations from above.
\end{Proposition}

\begin{proof}
We proceed as follows. Firstly, we find relations in $\oM_{2}$ using 3-spin structures, then
pull them back by forgetting $n$ points. Get again a PPZ relation and argue why what we get the full set of relations.
We are using the PPZ relations are stable under pulling back via forgetting morphism $\pi\colon \oM_{g,n+k}\to\oM_{g,n}$
\begin{align*}
 \pi^*T(g,n,r,a_1,\dots,a_n,d) = T(g,n+k,r,a_1,\dots,a_n,0,\dots,0,d).
\end{align*}

Now, the Picard group of $\oM_{2}$ is spanned on $\kappa_1$, $\delta_1$ and $\delta_{\rm irr}$, see Figure \ref{fig:SHHefigura}.

\begin{figure}\centering
\begin{subfigure}{0.25\textwidth}\centering
		\begin{tikzpicture}
		\node[draw,circle,minimum size=\rsmall] (0) at (0,0) {\( 1 \)} ;
		\path (0) edge [loop above] node {} (0);
		\end{tikzpicture}
		\renewcommand{\thesubfigure}{i}
		\caption{\( \delta_{\rm irr} \)}
		\end{subfigure}
		\begin{subfigure}{0.25\textwidth}\centering
		\begin{tikzpicture}
		\node[draw,circle,minimum size=\rsmall] (g) at (0,0) {\( 1 \)} ;
		\node[draw,circle,minimum size=\rsmall] (0) at (1.5,0) {\( 1\)};
		\path (g) edge (0);
		\end{tikzpicture}
		\renewcommand{\thesubfigure}{ii}
		\caption{\( \delta_{1} \)}
		\end{subfigure}
		\begin{subfigure}{0.25\textwidth}\centering
		\begin{tikzpicture}
		
		\node[draw,circle,minimum size=\rsmall] (g) at (0,0) {\( 2\)} ;
		
		\node[] (u) at (0.0,1.5) {};
		\draw (g) --(u) [dashed];
		\end{tikzpicture}
		\renewcommand{\thesubfigure}{iii}
		\caption{\( \kappa_1 \)}
		\end{subfigure}
\caption{Divisors in \texorpdfstring{$\oM_{2}$}{M 2}.}	\label{fig:SHHefigura}
\end{figure}

Witten's class has degree
$D^3_{2,0} = \frac{(r-2)}{r}$,
which is less than one, so we get a relation. Graphs in Figure \ref{fig:SHHefigura} give the following contributions:
\begin{gather*}
	-(r-1)^2\kappa_1 \big(R_1^{-1}\big)^{r-2}_0 = -(r-1)^2 \kappa_1 P_1(r,0),
\\
- \sum_{i,j=0}^{r-2}\omega_{1,1}(i)\omega_{1,1}(j)\big(R_1^{-1}\big)^{i}_{r-2-j}\delta_1 = -(r-1)^2 P_1(r,1) \delta_1	,
\qquad
	-(r-1)\delta_{\rm irr} \frac{2-3r+r^2}{24} .
\end{gather*}

When we put $r=3$, after simple algebra, we get
$5\kappa_1 = 7\delta_1 + \delta_{\rm irr}$.

Now we pullback this relation by forgetting $n$ points $\pi\colon \oM_{2,n} \rightarrow \oM_{2}$.
Using the well-known formulas for the pull-backs along the forgetful map
\begin{align*}
 \pi_i^*\kappa_1 = \kappa_1 - \psi_i,\qquad
 \pi_i^*\psi_j = \psi_j - \delta^0_{\{i,j\}},\qquad
 \pi_i^*\delta^h_I = \delta^h_I + \delta^h_{I\cup\{i\}},
\end{align*}
we get, after iterating $n$ times, the relation
\begin{align*}
	5\bigg(\kappa_1 - \sum \psi_i + \sum_{I\subset{1,\dots,n}}\delta_0^{I}\bigg) = \delta_{\rm irr}^{\prime} + 7\sum_{I\subset{1,\dots,n}}\delta_1^{I}.
\end{align*}
 Now suppose we have another PPZ relation $Q=0$ in degree~1. As we noted at the end of Section~\ref{tautrel}, this relation came from a certain choice of a tuple $(a_1,\dots, a_n)$ such that $\sum a_i = 0 \mod 2$ and \smash{$1>\frac{1+\sum a_i}{3}$}. Now we have only one possibility corresponding to the choice when all~$a_i$'s equal to 0. So, $Q$ is proportional to our relation. Moreover, this relation is the same as in Arbarello--Cornalba's \cite[Theorem 2.2]{Cor}. Other choices of $r\geq3$ would lead to the same relation since the set of Arbarello--Cornalba relations is full and PPZ relations hold in cohomology. So, we are done.
\end{proof}

\section{Higher genera and conclusion}

In genus 3 with no markings, we get zero contributions since the parity conditions for vertices are not compatible. We can not satisfy the conditions. For example, curves with a node separating genus~1 and genus~2 components give
\begin{align*}
-\delta_1\sum_{p,q=0}^1\omega_{1,1}(p)\omega_{1,1}(p)\omega_{2,1}(q)\big(R_1^{-1}\big)^p_{1-q}.
\end{align*}
So, $p$ and $q$ must be $p=0$, $q=1$, hence $\big(R_1^{-1}\big)^0_{0}$ is zero. We can not satisfy the conditions ${\sum a_i = 1 \mod 2}$ and \smash{$1>\frac{\sum a_i+2}{3}$} at all.

In genus 4 with no markings, we get the expression of the Witten's class instead of a tautological relation since
$D^3_{4,0} = 1$.
And in higher genera, the degree of Witten's class is bigger than 1 for all $r$. These considerations agree with early computations and do not contradict the fact that the only relations in $RH^2\big(\oM_{g,n},\mathbb{Q}\big)$ with $g\geq 3$ have the trivial form $\delta_h=\delta_{g-h}$, $h=0,\dots, g-1$.

From above, it is clear that the tautological PPZ relations for every particular $r$ are in fact the full set relations in $H^2\big(\oM_{g,n},\mathbb{Q}\big)$
 \begin{align*}
RH^2\big(\oM_{g,n},\mathbb{Q}\big) = H^2\big(\oM_{g,n},\mathbb{Q}\big),
\end{align*}
which confirms Pixton's hypothesis in this codimension.

\subsection*{Acknowledgements}
The author is very thankful to P.~Dunin-Barkowski for numerous discussions and competent advisoring during all the project. I am also grateful to D.~Zvonkine for teaching me a trick used in genus 2 case which led to a significant simplification and for helpful remarks. I thank anonymous referees for numerous tips. The author is partially supported by International Laboratory of Cluster Geometry NRU HSE, RF Government grant, ag. no.~075-15-2021-608 dated 08.06.2021.

\pdfbookmark[1]{References}{ref}
\LastPageEnding

\end{document}